\documentclass[12pt]{amsart}

\usepackage{fullpage}
\usepackage{setspace}
\usepackage[leqno]{amsmath}
\usepackage{amsfonts}
\usepackage{amssymb}
\usepackage{amsthm}
\usepackage{amssymb}
\usepackage{longtable}
\usepackage{bbm}

\theoremstyle{plain}
\newtheorem{thm}{Theorem}[section]

\newtheorem{lem}[thm]{Lemma}

\theoremstyle{definition}

\newcommand{\R}{\mathbb{R}}
\newcommand{\C}{\mathbb{C}}
\newcommand{\ga}{\gamma}

\newcommand{\be}{\beta}

\newcommand{\si}{\sigma}

\newcommand{\al}{\alpha}

\newcommand{\Xa}{X_\alpha}
\newcommand{\xa}{x_\alpha}
\newcommand{\xb}{x_\beta}
\newcommand{\xc}{x_\gamma}

\newcommand{\Ra}{R^\alpha}
\newcommand{\Rb}{R^\beta}
\newcommand{\La}{L^\alpha}
\newcommand{\Lb}{L^\beta}
\DeclareMathOperator{\nr}{nilrad}
\DeclareMathOperator{\Ann}{Ann}
\newcommand{\ns}{N\!S}
\newcommand{\eL}{\mathfrak{L}}

\title{Solvable Leibniz Algebras with Abelian Nilradicals}

\author[Bosko-Dunbar]{Lindsey Bosko-Dunbar}
\address{Department of Mathematics, Spring Hill College\\
Mobile, AL 36608}
\email{lboskodunbar@shc.edu}

\author[Burke]{Matthew Burke}
\address{Department of Mathematics, Spring Hill College\\
Mobile, AL 36608}
\email{mjburke@email.shc.edu}

\author[Dunbar]{Jonathan D. Dunbar}
\address{Department of Mathematics, Spring Hill College\\
Mobile, AL 36608}
\email{jdunbar@shc.edu}

\author[Hird]{J.T. Hird}
\address{Department of Mathematics, West Virginia University, Institute of Technology\\
Montgomery, WV 25136}
\email{John.Hird@mail.wvu.edu}

\author[Stagg Rovira]{Kristen Stagg Rovira}
\address{Department of Mathematics, The University of Texas at Tyler\\
Tyler, TX 75799}
\email{kstagg@uttyler.edu}

\begin{document}
\doublespacing

\begin{abstract}
We extend the classification of solvable Lie algebras with abelian nilradicals to classify solvable Leibniz algebras which are one dimensional extensions of abelian nilradicals.
\end{abstract}

\maketitle

\section{Introduction}\label{intro}

    Leibniz algebras were defined by Loday in 1993 \cite{loday, loday2}.  Recently, there has been a trend to show how various results from Lie algebras extend to Leibniz algebras \cite{ao, ayupov, omirov}.  In particular, there has been interest in extending classifications of certain classes of Lie algebras to classifications of corresponding Leibniz algebras \cite{aor, chelsie-allison, bdhs, clok, clok2, kho}.

Some of the classifications of Lie algebras come from placing certain restrictions on the nilradical \cite{cs, nw, tw, wld}.  Several authors have been able to extend these results to Leibniz algebras or show similar results for Leibniz algebras with certain restrictions on the nilradical \cite{bdhs, clok, clok2, kko, kho}.  In praticular, Ndogmo and Winternitz \cite{nw} study solvable Lie algebras with abelian nilradicals.  The goal of this paper is to utilize the results of \cite{nw} to develop similar results in the Leibniz setting. 

We construct a general classification theorem for solvable Leibniz algebras with abelian nilradicals over $\C$. Furthermore, we discuss the case of 1-dimensional extensions; then provide an explicit classification of all 1-dimensional extensions of 1-, 2-, and 3-dimensional abelian nilradicals.  In \cite{ck}, Ca\~nete and Khudoyberdiyev classify all non-nilpotent 4-dimensional Leibniz algebras over $\C$. Our classification recovers their result for 4-dimensional algebras with 3-dimensional abelian nilradicals.  We also develop some results on 2-dimensional extensions of abelian nilradicals, specifically we show that all 4-dimensional solvable Leibniz algebras with  2-dimensional abelian nilradicals are in fact Lie algebras.

\section{Preliminaries}

A Leibniz algebra, $\eL$, is a vector space over a field (which we will take to be $\C$ or $\R$) with a bilinear operation (which we will call multiplication) defined by $[x,y]$ which satisfies the Jacobi identity
\begin{equation}\label{Jacobi}
[x,[y,z]] = [[x,y],z] + [y,[x,z]]
\end{equation}
for all $x,y,z \in \eL$.  In other words $L_x$, left-multiplication by $x$, is a derivation.  Some authors choose to impose this property on $R_x$, right-multiplication by $x$, instead.  Such an algebra is called a ``right'' Leibniz algebra, but we will consider only ``left'' Leibniz algebras (which satisfy \eqref{Jacobi}).  $\eL$ is a Lie algebra if additionally $[x,y]=-[y,x]$.

The derived series of a Leibniz (Lie) algebra $\eL$ is defined by $\eL^{(1)}=[\eL,\eL]$, $\eL^{(n+1)}=[\eL^{(n)},\eL^{(n)}]$ for $n\ge 1$.  $\eL$ is called solvable if $\eL^{(n)}=0$ for some $n$. The lower-central series of $\eL$ is defined by $\eL^2 = [\eL,\eL]$, $\eL^{n+1}=[\eL,\eL^n]$ for $n>1$. $\eL$ is called nilpotent if $\eL^n=0$ for some $n$.  It should be noted that if $\eL$ is nilpotent, then $\eL$ must be solvable.

The nilradical of $\eL$ is defined to be the (unique) maximal nilpotent ideal of $\eL$, denoted by $\nr(\eL)$.  It is a classical result that if $\eL$ is solvable, then $\eL^2 = [\eL,\eL] \subseteq \nr(\eL)$.  From \cite{mubar}, we have that
\begin{equation*}\label{dimension}
\dim (\nr(\eL)) \geq \frac{1}{2} \dim (\eL).
\end{equation*}
An abelian Lie algebra $A(r)$ is the $r$-dimensional Lie algebra with basis \newline $\{n_1, n_2, \ldots, n_r\}$, and multiplication
\begin{equation}\label{Abel}
[n_i,n_j]=0,\ \forall\ i,j=1,2,\ldots,r.
\end{equation}

The left-annihilator or left-normalizer of a Leibniz algebra $\eL$ is the ideal $\Ann_\ell(\eL) = \left\{x\in \eL\mid [x,y]=0\ \forall y\in \eL\right\}$. Note that the elements $[x,x]$ and $[x,y] + [y,x]$ are in $\Ann_\ell(\eL)$, for all $x,y\in \eL$, because of \eqref{Jacobi}.

An element $x$ in a Leibniz algebra $\eL$ is nilpotent if both $(L_x)^n = (R_x)^n = 0$ for some $n$.  In other words, for all $y$ in $\eL$
\begin{equation*}
[x,\cdots[x,[x,y]]] = 0 = [[[y,x],x]\cdots,x].
\end{equation*}

A set of matrices $\{\Xa\}$ is called linearly nilindependent if no non-zero linear combination of them is nilpotent.  In other words, if
\begin{equation*}
X = \displaystyle\sum_{\al=1}^f c_\al \Xa,
\end{equation*}
then $X^n=0$ implies that $c_\al=0$ for all $\al$.  A set of elements of a Leibniz algebra $\eL$ is called linearly nilindependent if no non-zero linear combination of them is a nilpotent element of $\eL$.

\section{Classification}

Let $A(r)$ be the $r$-dimensional abelian (Lie) algebra over the field $F$ ($\C$ or $\R$) with basis $\{n_1, \ldots, n_r\}$ and products given by \eqref{Abel}.  By appending the basis of $A(r)$ with $s$ linearly nilindependent elements $\{x_1, \ldots, x_s\}$, we construct an $n$-dimensional solvable Leibniz algebra, $L(r,s)$ where $n=r+s$.  In doing so, we create a Leibniz algebra whose nilradical is $A(r)$.  Henceforth, we shall only consider such $L(r,s)$ that are indecomposable. As in \cite{nw}, we have the following constraints on $r$, $s$, and $n$:
\begin{equation}\label{bounds}
n=r+s,\quad \dfrac{n}{2}\le r \le n-1.
\end{equation}

Let $\eL$ be a finite-dimensional solvable non-nilpotent Leibniz algebra over the field $F$, of characteristic zero, subject to \eqref{bounds}.  Let $\eL$ have an $r$-dimensional abelian nilradical, namely $\nr(L)=A(r)$. For $\eL$, we may choose the basis $\{n_1, \ldots, n_r, x_1,\ldots,x_s\}$, where $n_1, \ldots, n_r\in A(r)$ and $x_1, \ldots, x_s\in \eL\backslash A(r)$. In general, the bracket relations for elements in $A(r)$ are defined by $[n_i,n_j]=0$, $\forall\ n_i,n_j\in A(r)$. Letting $N = (n_1,\ n_2,\ \ldots,\ n_r)^T$, then the left and right bracket relations of elements $x_\al$ 
are defined, for $1\le\al\le s\le\dfrac n2$, by
\begin{align}
	\tag{4a}\label{xnproducts}
    \begin{split}
    [x_\al,N]&= \La N,\quad \La\in F^{r\times r},\\ 
    [N,x_\al]&= \Ra N,\quad \Ra\in F^{r\times r},
    \end{split}\\
    \tag{4b}\label{xproducts}
    [x_\al,x_\be] &= \si^{\al\be}_j n_j,\ \si^{\al\be}_j\in F,
\end{align}
    \addtocounter{equation}{1}
since $\eL^2\subseteq\nr(\eL)=A(r)$. Note that we are employing Einstein notation in \eqref{xproducts}, summing over all $j$.

The classification of these Leibniz algebras $L(r,s)$ is equivalent to the classification of the matrices $\La$ and $\Ra$ and the constants $\si^{\al \be}_{j}$.  The Jacobi identities for the triples $\{x_\al, x_\be, n_i\}$, $\{x_\al, n_{i}, x_\be\}$, $\{n_{i}, x_\al, x_\be\}$ for $1 \leq \al, \be \leq s$ and $1 \leq i \leq r$ yield, respectively,
\begin{align*}
(\La_{ij}\Lb_{jk} - \Lb_{ij}\La_{jk}) n_k =&	0\\
(\Rb_{ij}\La_{jk} - \La_{ij}\Rb_{jk}) n_k =&	0\\
(\Ra_{ij}\Rb_{jk} + \Rb_{ij}\La_{jk}) n_k =&	0.
\end{align*}
Unlike the Lie case, these give nontrivial relations for $s=1$ or $\al=\be$ when $s>1$.  By the linear nilindependence of the $n_k$, we have the following relations on the matrices $\La$ and $\Ra$
\begin{align}
\tag{5a}\label{2.7}[\La,\Lb]	=&	0\\
\tag{5b}\label{2.7twist}[\La,\Rb] =&	0\\
\tag{5c}\label{un2.7}(\Ra + \La)\Rb =&	0,
\end{align}
\addtocounter{equation}{1}
where \eqref{un2.7} utilizes \eqref{2.7twist}.

The Jacobi identity for the triple $\{x_\al, x_\be, x_\ga\}$ for $1 \leq \al, \be, \ga \leq s$ gives us
\begin{equation}\label{2.16}
\si^{\be \ga}_{j} \La_{jk} - \si^{\al \be}_{j} R^\ga_{jk} - \si^{\al \ga}_{j} \Lb_{jk} = 0,
\end{equation}
summed over $j$.  Again, we do not require $\al, \be, \ga$ to be distinct, which in particular gives nontrivial relations for $s \geq 1$.  Equivalently, since $[x_\al,x_\al]\in \Ann_\ell(L)$, we have $\si^{\al\al}_{j}\Rb_{jk}=0$ $\forall\ \al,\be=1,\ldots,s$ and, again, we are summing over index $j$.


\begin{lem}\label{nullspace}
Let $L$ be a Leibniz algebra with abelian nilradical, with elements defined by \eqref{xnproducts} and \eqref{xproducts}, and denote $\si^{\al\be}=\begin{pmatrix}	\si^{\al\be}_1	\\	\vdots	\\	\si^{\al\be}_r	\end{pmatrix}$. Then, $\si^{\al\al},\si^{\al\be}+\si^{\be\al}\in\ns((R^\gamma)^T)$, for all $\gamma=1,\ldots,s$.
\end{lem}
\begin{proof}
    Recall that, for Leibniz algebra $L$, $[x,y]+[y,x]\in \Ann_\ell(L)$, $\forall x,y\in L$. Hence, $[\xa,\xb]+[\xb,\xa]\in \Ann_\ell(L)$,\ $\forall \al,\be=1,\ldots,s$. In particular, $[[\xa,\xb]+[\xb,\xa],\xc]=0$, which implies that
    $$(\si^{\al\be}_i+\si^{\be\al}_i)R^\gamma_{ij}=0,$$
    where here we are summing over all $i=1,\ldots,r$. Additionally, this is true for all $j=1,\ldots,r$ and all $\gamma=1,\ldots,s$.  Thus, $(\si^{\al\be}+\si^{\be\al})^T R^\gamma=0$, which implies that $(R^\gamma)^T(\si^{\al\be}+\si^{\be\al})=0$.  Therefore, $\si^{\al\be}+\si^{\be\al}\in\ns((R^\gamma)^T)$.  What is more, if $\be=\al$, then it immediately follows that $\si^{\al\al}\in\ns((R^\gamma)^T)$, as well.
\end{proof}
It is advisable to note that Lemma \ref{nullspace} does not imply that $\si^{\al\be}+\si^{\be\al}$ or $\si^{\al\al}$ need be nontrivial elements of $\ns((R^\gamma)^T$.

In an effort to simplify matrices $\La$ and $\Ra$ and constants $\si^{\al\be}_j$, we employ several transformations which leave bracket relations \eqref{Abel} invariant. Namely,
	\begin{itemize}
    \item	redefine the elements of $\nr(L)$:
    	\begin{equation}\label{Stransform}
		\begin{split}
		\hspace{17pt}N \longrightarrow SN, \quad S \in GL(r,F) \\
		\Rightarrow
		\begin{cases}
		\Ra \longrightarrow S\Ra S^{-1}\\
		\La \longrightarrow S\La S^{-1},
		\end{cases}
		\end{split}
		\end{equation}
    
    \item	redefine the elements of the extension:
    	\begin{equation}\label{xtransform}
        	x_\al \longrightarrow x_\al + \mu^\al_j n_j,\quad \mu^\al_j\in F,
        \end{equation}
    
    \item	redefine the extension by linear combinations of $x_\al$:
    	\begin{equation}\label{Gtransform}
		\hspace{17pt}X \longrightarrow GX, \quad X \in GL(s,F),
		\end{equation}
where $X=(x_1,x_2,\ldots,x_s)^T$.    
    \end{itemize}

The matrix $S$ used in \eqref{Stransform} is simply required to be invertible to preserve \eqref{Abel}.  So, we may choose $S$ to be the appropriate permutation matrix which transforms $R^1$ into Jordan canonical form.  Therefore, the classification of Leibniz algebras $L(r,s)$ with abelian nilradical will rely primarily on classes defined by the Jordan canonical form of the $r\times r$ matrix $R^1$.  Unfortunately, it cannot be guaranteed that the same $S$ will also transform $L^1$ or, $\Ra$ or $\La$, $\al>1$, into Jordan canonical form.  These matrices, however, will be determined by other constraints.  For example, note that if $R^1$ has
  no zero eigenvalues, then $(R^1)^{-1}$ exists, and so by \eqref{un2.7}, $R^\al=-L^\al$ for all $\al=1,\ldots,s$.

Observe that, in \eqref{Gtransform}, for the special case of a 1-dimensional extension of $A(r)$, the matrix $G$ will be a scalar. 

Note that \eqref{xtransform} leaves $\ns((\Ra)^T)$ and the matrices $\La$ and $\Ra$ invariant, but it will have an effect on $\si^{\al\be}_k$.  That is,
	\begin{equation}\label{sigtransform}
    	\si^{\al\be}_k\longrightarrow \si^{\al\be}_k + \mu^{\al}_j\Rb_{jk} + \mu^\be_j\La_{jk},\quad \forall\ k=1,\ldots,r.
    \end{equation}
In the special case where $\al=\be$, we see that \eqref{sigtransform} appears as $\si^{\al\al}_k\longrightarrow \si^{\al\al}_k + \mu^{\al}_j(\Ra_{jk} + \La_{jk})$, implying that we cannot use \eqref{xtransform} to manipulate $\si^{\al\al}$ when $\Ra=-\La$, for $\al=1,\ldots,s$.  This is of particular value in the $s=1$ case, as it implies that we may not change $\si^{11}_k = \si_k$ when $R^1=-L^1$.
        
\section{1-Dimensional Extensions of $A(r)$}    
We will explore the non-Lie Leibniz algebras whose abelian nilradical $A(r)$ is extended by a single semisimple element, $x$.  In this case we have a specialized version of the previous result Lemma \ref{nullspace}, as well as a general classification theorem for non-Lie Leibniz algebras $L(r,1)$.  Throughout this section, we will assume that $R$ is transformed by \eqref{Stransform} into Jordan canonical form, which places the eigenvalues of $R$ on its diagonal.  Note that at least one of these eigenvalues must be nonzero, or else $x$ would act nilpotently on $A(r)$ from the right.  Once this transformation is complete, we will reorder the basis elements of $A(r)$ such that any Jordan blocks associated with zero eigenvalues are moved to the bottom right corner of $R$.  We do this reordering in such a way so as to preserve each Jordan block.  Lastly, we perform transformation \eqref{Gtransform}.  Since $s=1$, $G$ is any nonzero scalar of our choosing.  Let $G=\dfrac{1}{\lambda_1}$, where $\lambda_1$ is the first eigenvalue of $R$, associated with the nilradical basis element $n_1$.  This ensures that $R_{11}=1$.  Putting the result back in Jordan form preserves the 1's off the main diagonal.

\begin{lem}\label{Lr1nullspace}
	Let $L(r,1)$ be a Leibniz algebra that is a 1-dimensional extension of an $r$-dimensional abelian nilradical, with elements defined by \eqref{xnproducts} and \eqref{xproducts}, and denote $\si= (\si_1,\ \ldots,\ \si_r)^T$. Then, 
	\begin{enumerate}
    	\item	$\si\in\ns(R^T)$, and 
        \item	
        		$R_{ij}\ne0$ implies that $\si_i=0$.
	\end{enumerate}
\end{lem}
\begin{proof}
Statement (1)  follows directly from Lemma \ref{nullspace}, in the case of a 1-dimensional extension.

For statement (2), recall that we may choose the appropriate permutation matrix $S$ which transforms $R$ into Jordan canonical form by conjugation.  Once this transformation is applied, $R$ consists of Jordan blocks with each nonzero column of $R$ either having a single nonzero entry or two nonzero entries. Suppose that column $j$ has a single nonzero entry in row $i$.  Then, by \eqref{2.16}, we know that $\si_i=0$.

Suppose now that column $j$ has two nonzero entries.  Then, these entries must be $R_{jj}$ and $R_{j-1,j}$, since they are components of $J_\lambda$, an $m\times m$ Jordan block associated with the eigenvalue $\lambda\ne0$.  (Hence $i=j$ or $j-1$.)  So, there is a smallest integer $p<m$ such that the $j-p$ column of $R$ has a single nonzero entry, namely $R_{(j-p)(j-p)}$, the first nonzero entry of $J_\lambda$, and hence $\si_{j-p}=0$. Now consider the $j-(p-1)$ column of $R$, which must have two nonzero entries $R_{(j-p)(j-(p-1))}$ and $R_{(j-(p-1))(j-(p-1))}$.  By \eqref{2.16}, it must be the case that $R_{(j-p)(j-(p-1))}\si_{j-p} + R_{(j-(p-1))(j-(p-1))}\si_{j-(p-1)}=0$. Since $\si_{j-p}=0$ and $R_{(j-(p-1))(j-(p-1))}=\lambda\ne0$, then $\si_{j-(p-1)}=0$.  An iteration of this process will yield that $\si_{j-1}=\si_j=0$, so $\si_i=0$ as needed.
\end{proof}

\begin{lem}\label{Lr1Lie}
Let $L(r,1)$ be a solvable Leibniz algebra that is a 1-dimensional extension by $x$of an $r$-dimensional abelian nilradical, $A(r)$.  If the matrix $R$, which defines the right-action of $x$ on $A(r)$, is nonsingular, then $L(r,1)$ is a Lie algebra.
\end{lem}
\begin{proof}
	If $R$ is nonsingular, then $R^{-1}$ exists.  So, by \eqref{un2.7}, $R=-L$.  Furthermore, since we may use \eqref{Stransform} to transform $R$ into its Jordan canonical form, then the diagonal of $R$ will consist of its eigenvalues.  Nonsingular matrices have only nonzero eigenvalues. Thus, by Lemma \ref{Lr1nullspace}, $\si_i=0$, for all $i=1,\ldots,r$.  Therefore, $L(r,1)$ is a Lie algebra.
\end{proof}

\begin{thm}
Let $\eL$ be an $(r+1)$-dimensional solvable non-nilpotent Leibniz algebra over the field $F$, of characteristic zero with $r<\infty$. Let $\eL$ have the $r$-dimensional abelian nilradical, namely $\nr(L)=A(r)$, and let $\eL$ be subject to \eqref{bounds}. We may choose the basis of $L$ to be $\{n_1, \ldots, n_r, x\}$, where $n_1, \ldots, n_r\in A(r)$ and $x \in L\backslash A(r)$. The bracket relations for elements in $A(r)$ are defined by $[n_i,n_j]=0$, $\forall\ n_i,n_j\in A(r)$. Letting $N = (n_1, n_2,\ldots, n_r)^T$, then the left and right bracket relations of x 
are defined by
\begin{align}
	\tag{11a}\label{xnproduct}
    \begin{split}
    [x,N]&= LN,\quad L \in F^{r\times r},\\ 
    [N,x]&= RN,\quad R \in F^{r\times r},
    \end{split}\\
    \tag{11b}\label{xbracket}
    [x,x] &= \si_j n_j,\ \si_j\in F, \text{ such that }\si\in\ns(R^T),
\end{align}
\addtocounter{equation}{1}
where $\sigma=(\sigma_1,\ \ldots,\ \sigma_r)^T$
\end{thm}

We will specifically describe in detail $L(r,1)$ in the cases of $r = 1,2,3$.

\subsection{Leibniz algebras $L(1,1)$ of non-Lie type}

For $L(1,1) = \langle n,x\rangle$, the products are $[n,x]=Rn$, $[x,n]=Ln$, and $[x,x]=\si n$, with $R,L,\si\in F$. By Lemma \ref{Lr1Lie}, we see that if $R\ne 0$, then $L(1,1)$ is a Lie algebra. Suppose, then, that $R=0$, implying that $L\ne0$, else $x\in\nr(L(1,1))$. Thus, using \eqref{Gtransform} with $G=L^{-1}$, we obtain the algebra described in Table \ref{L11}.  Note that since $R=0$, there is no restriction placed on $\si$.  These algebras are classified in section 1 of Table \ref{L11}.

\subsection{Leibniz algebras $L(2,1)$ of non-Lie type}

For $r = 2$, the only two Jordan canonical forms are
\begin{align}
	\tag{12a}\label{JCF1}
    &\left(\begin{array}{cc}
    	\lambda_1 &0\\
        0&\lambda_2\\
    \end{array}\right)\\
    \tag{12b}\label{JCF2}
    &\left(\begin{array}{cc}
    	\lambda_1&1\\
        0&\lambda_1\\
    \end{array}\right)
\end{align}
\addtocounter{equation}{1}

     
where $\lambda_1$ and $\lambda_2$ are eigenvalues of the matrix, $R$, that we are transforming, where $\lambda_1$ and $\lambda_2$ are not necessarily distinct. By Lemma \ref{Lr1Lie}, $R$ must be singular.
This leaves us with three possibilities for $R$:
$$
\left(\begin{array}{cc}
    	0&0\\
        0&0\\
    \end{array}\right), \qquad
\left(\begin{array}{cc}
    	\lambda_1 &0\\
        0&0\\
    \end{array}\right), \qquad
\left(\begin{array}{cc}
    	0&1\\
        0&0\\
    \end{array}\right).
$$

\subsubsection{Case 1: $R=0$.}
Since $R=0$, \eqref{Stransform} leaves $R$ invariant (the zero matrix is always in Jordan form) and so can be used to put $L$ in Jordan form instead, so $L$ will be of form \eqref{JCF1} or \eqref{JCF2}.  If the diagonals of $L$ are zero ($\lambda_1=\lambda_2=0$), then $L$ is nilpotent, hence so is $\eL$, which contradicts $\nr(\eL)=A(2)$.

Using \eqref{Gtransform} we can transform $\lambda_1$ to 1 in either case, so $L$ is one of
$$
\left(\begin{array}{cc}
    	1&0\\
        0&\lambda\\
    \end{array}\right), \qquad
\left(\begin{array}{cc}
    	1 &1\\
        0&1\\
    \end{array}\right),
$$
where $\lambda=\lambda_2$ can be zero (note that if $\lambda_1=0$ and $\lambda_2 \neq 0$ then we must first switch the basis elements $n_1$ and $n_2$).  Now applying \eqref{xtransform} with the appropriate choice of $\mu$ makes $\sigma_1=\sigma_2=0$ except when $\lambda=0$, in which case $\sigma_1=0$ and $\sigma_2$ is free.  
These algebras are classified in section 2 of Table \ref{L21}.

 \subsubsection{Case 2:} $R=
\left(\begin{array}{cc}
     	\lambda_1 &0\\
         0&0\\
     \end{array}\right)$.
     
Using \eqref{Gtransform} with $G=\dfrac{1}{\lambda_1}$, we can further simplify this matrix to $R = 
\left(\begin{array}{cc}
    	1	&0\\
        0	&0
    \end{array}\right)$.
Applying relations \eqref{2.7twist} and \eqref{un2.7} we find that the matrix $L$ is of the form $L =
\left(\begin{array}{cc}
    	-1	&0\\
        0	&a
    \end{array}\right)$.

If $a=0$, then \eqref{sigtransform} implies that \eqref{xtransform} leaves $\si$ invariant, so $\si_1$ and $\si_2$ are free.  However, if $\si_1=\si_2=0$, then $\eL$ is Lie.  If $a \neq 0$, 
then an appropriate choice of $\mu$ makes $\si_2=0$, 
but $R_{11} = -L_{11}$, so \eqref{sigtransform} implies that \eqref{xtransform} leaves $\si_1$ invariant.  These algebras are classified in section 2 of Table \ref{L21}.

\subsubsection{Case 3:} $R=
\left(\begin{array}{cc}
    	0 &1\\
        0&0\\
    \end{array}\right)
$.\label{maxnilp}
Applying relations \eqref{2.7twist} and \eqref{un2.7} we find that the matrix $L$ is of the form $L =
\left(\begin{array}{cc}
    	0	&a\\
        0	&0
    \end{array}\right)$.
Thus $R$ and $L$ are both nilpotent, hence so is $x$, which contradicts $\nr(\eL)=A(2)$.  Therefore we find no permissible algebras in this case.




\subsection{Leibniz algebras $L(3,1)$ of non-Lie type}

Solvable Leibniz algebras $L(3,1)$ of non-Lie type are 1-dimensional extensions of a 3-dimensional abelian nilradical. Similar to the previous section, we begin classification by first considering the possible Jordan canonical forms for $3\times 3$ matrices. 
Again, we arrange the eigenvalues so that the non-zero eigenvalues, if there are any, come before zero eigenvalues. The possible Jordan canonical forms that matrix $R$ may take are as follows:
\begin{gather}
	\tag{14a}\label{JCF3}
    \begin{pmatrix}
    	\lambda_1	&	0			&	0	\\
        0			&	\lambda_2	&	0	\\
        0			&	0			&	\lambda_3
	\end{pmatrix}\phantom{,}\\
    \tag{14b}\label{JCF4}
    \begin{pmatrix}
    	\lambda_1	&	1			&	0	\\
        0			&	\lambda_1	&	0	\\
        0			&	0			&	\lambda_2
	\end{pmatrix}\phantom{,}
\qquad
    \begin{pmatrix}
    	\lambda_1	&	0			&	0	\\
        0			&	\lambda_2	&	1	\\
        0			&	0			&	\lambda_2
	\end{pmatrix}\phantom{,}\\
    \tag{14c}\label{JCF5}
    \begin{pmatrix}
    	\lambda_1	&	1			&	0	\\
        0			&	\lambda_1	&	1	\\
        0			&	0			&	\lambda_1
	\end{pmatrix},
\end{gather}
\addtocounter{equation}{1}
where $\lambda_1$, $\lambda_2$, and $\lambda_3$ are not necessarily distinct.

As before, by Lemma \ref{Lr1Lie}, $R$ must be singular. This means that we can force at least one $\lambda_i$ to be zero.  If $\lambda_1 \neq 0$, then once we perform transformation \eqref{Gtransform} with $G=\dfrac{1}{\lambda_1}$, the leading eigenvalue becomes 1.  Putting the result back in Jordan form, the following possibilities for $R$ remain:
\begin{gather}
	\label{L31nilp}
    \begin{pmatrix}
    	0\	&	0\	&	0\	\\
        0\	&	0\	&	0\	\\
        0\	&	0\	&	0\
	\end{pmatrix}
\qquad
    \begin{pmatrix}
    	0\	&	1\	&	0\	\\
        0\	&	0\	&	0\	\\
        0\	&	0\	&	0\
	\end{pmatrix}
\qquad
    \begin{pmatrix}
    	0\	&	1\	&	0\	\\
        0\	&	0\	&	1\	\\
        0\	&	0\	&	0\
	\end{pmatrix}
\\
    \begin{pmatrix}
    	1\	&	0\	&	0\	\\
        0\	&	0\	&	0\	\\
        0\	&	0\	&	0\
	\end{pmatrix}
\qquad
	\begin{pmatrix}
    	1\	&	0\	&	0\	\\
        0\	&	a\	&	0\	\\
        0\	&	0\	&	0\
	\end{pmatrix}
\qquad
    \begin{pmatrix}
    	1\	&	1\	&	0\	\\
        0\	&	1\	&	0\	\\
        0\	&	0\	&	0\
	\end{pmatrix}
\qquad
    \begin{pmatrix}
    	1\	&	0\	&	0\	\\
        0\	&	0\	&	1\	\\
        0\	&	0\	&	0\
	\end{pmatrix},\label{L31sing}
\end{gather}
where $a\ne0 \in F$.

\subsubsection{Case 1: $R=0$.}
As before, since $R=0$ we can use \eqref{xtransform} to put $L$ in Jordan form.  If $L$ is nilpotent (one of the matrices in \eqref{L31nilp}) then $X$ is nilpotent, which contradicts $\nr(\eL)=A(3)$.  Thus we can assume $L_{11} \neq 0$, and we can use \eqref{Gtransform} to make $L_{11}=1$.  This means that $L$ can be written in the form of one of the matrices in \eqref{JCF3}, \eqref{JCF4}, or \eqref{JCF5} with $\lambda_1=1$.  The resulting algebras are classified in Table \ref{L31}.

\subsubsection{Case 2:} $R=\begin{pmatrix} 0&1&\\ &0&\\ &&0 \end{pmatrix}$.
Using relations \eqref{2.7twist} and \eqref{un2.7} we have that matrix $L$ is of the form $L=\begin{pmatrix} 0&a&b \\ 0&0&0 \\ 0&c&d \end{pmatrix}$.  If $d=0$, then $L^3=0$ and $X$ is nilpotent, which is a contradiction.  On the contrary, if $d \neq 0$, then $(L^n)_{33} \neq 0$ for all $n$, so $X$ is not nilpotent.  Thus since $d \neq 0$ we can use \eqref{xtransform} to make $d=1$.  If $c\neq 0$, then we can use \eqref{Stransform} with $S=\begin{pmatrix} 1&&\\ &1&\\ &&1/c \end{pmatrix}$ to make $c=1$.  Note that we can use \eqref{xtransform} to make $\si_1=0$, and unless $a=-1$ and $c=0$ we can also make $\si_2=0$.  The resulting algebras are classified in Table \ref{L31}.

\subsubsection{Case 3:} $R=\begin{pmatrix} 0&1&\\ &0&1\\ &&0 \end{pmatrix}$.
Relations \eqref{2.7twist} and \eqref{un2.7} give that matrix $L$ is of the form $L=\begin{pmatrix} 0&-1&a \\ 0&0&-1 \\ 0&0&0 \end{pmatrix}$.  Thus both $R$ and $L$ are nilpotent, hence so is $X$, which contradicts $\nr(\eL)=A(3)$.  Thus we find no new algebras in this case (just as in section \ref{maxnilp}).

\subsubsection{Case 4:} $R=\begin{pmatrix} 1&&\\ &0&\\ &&0 \end{pmatrix}$.
Using relations \eqref{2.7twist} and \eqref{un2.7} we find that the matrix $L$ is of the form $L = \begin{pmatrix} -1&& \\ &a&b \\ &c&d \end{pmatrix}$.  At this point we can change the basis of $\text{span}\{n_2, n_3\}$ to put $\begin{pmatrix} a&b \\ c&d \end{pmatrix}$ in Jordan form.  This leaves us with two possibilities for $L$:
$$\begin{pmatrix} -1&& \\ &a& \\ &&b \end{pmatrix}
\qquad
\begin{pmatrix} -1&& \\ &a&1 \\ &&a \end{pmatrix},$$
where $a$ and $b$ are not necessarily distinct.  If $a$ or $b$ is nonzero, then by \eqref{sigtransform} an appropriate choice of $\mu$ in \eqref{xtransform} will make the corresponding $\sigma$ be zero.  Additionally, if $a=0$ in the second case, choosing $\mu_2=-\si_3$ makes $\si_3=0$.  If $a=b=0$ in the first case ($L=-R$), then our only constraint on $\si$ is that $\si_1$, $\si_2$, and $\si_3$ are not all zero (or else $\eL$ is Lie).  However in this case, we are free to interchange $n_2$ and $n_3$ (hence $\si_2$ and $\si_3$), so we can assume $\si_2 \geq \si_3$.  The resulting algebras are classified in Table \ref{L31}.

\subsubsection{Case 5:} $R=\begin{pmatrix} 1&&\\ &a&\\ &&0 \end{pmatrix}$.
Here we assume $a \neq 0$.  Using relations \eqref{2.7twist} and \eqref{un2.7} we find that the matrix $L$ is of the form $L = \begin{pmatrix} -1&& \\ &-a& \\ &&b \end{pmatrix}$.
If $b \neq 0$ we can use \eqref{xtransform} to make $\si_3=0$, but since $R_{11}+L_{11}=R_{22}+L_{22}=0$ by \eqref{sigtransform} we can not force $\si_1$ or $\si_2$ to be zero (nor $\si_3$ if $b=0$).  The resulting algebras are classified in Table \ref{L31}.

\subsubsection{Case 6:} $R=\begin{pmatrix} 1&1&\\ &1&\\ &&0 \end{pmatrix}$.
Using relations \eqref{2.7twist} and \eqref{un2.7} we find that the matrix $L$ is of the form $L = \begin{pmatrix} -1&-1& \\ &-1& \\ &&a \end{pmatrix}$.  If $a\neq 0$ we can use \eqref{xtransform} to make $\si_3=0$, but by \eqref{sigtransform} we can not make $\si_1$ or $\si_2$ be zero.  The resulting algebras are classified in Table \ref{L31}.

\subsubsection{Case 7:} $R=\begin{pmatrix} 1&&\\ &0&1\\ &&0 \end{pmatrix}$.
 Using relations \eqref{2.7twist} and \eqref{un2.7} we find that the matrix $L$ is of the form $L = \begin{pmatrix} -1&& \\ &0&a \\ &&0 \end{pmatrix}$.  If $a\neq -1$ we can use \eqref{xtransform} to make $\si_3=0$, but by \eqref{sigtransform} we can not make $\si_1$ or $\si_2$ (or $\si_3$ if $a=-1$) be zero.  The resulting algebras are classified in Table \ref{L31}.

\section{2-Dimensional Extensions of $A(2)$}
The case of $L(r,s)$, which are $(r+s)$-dimensional Leibniz algebras with an $r$-dimensional abelian nilradical, seems significantly more complex when $s>1$.  This is also true in the Lie case, as seen in \cite{nw}.  So, we narrow our focus to the case of 2-dimensional extensions of 2-dimensional abelian nilradicals.

The Leibniz algebra $L(2,2)$ has basis $\{x_1,x_2,n_1,n_2\}$, satisfying relations \eqref{Abel}, \eqref{xnproducts}, and \eqref{xproducts}. As before, we may select $S$ so that $R^1$ is transformed into Jordan canonical form.  Therefore, $R^1$ will be of the form \eqref{JCF1} or \eqref{JCF2}.  Once $R^1$ is determined, the structure of $R^2$ can likewise be determined based on other restrictions, namely \eqref{2.7}, \eqref{2.7twist}, and \eqref{un2.7}.  In particular, \eqref{un2.7} guarantees that if $R^1$ is nonsingular, then $R^1 = -L^1$ and $R^2=-L^2$.  
\subsection{$R^1$ not diagonal}\label{case1}
We first consider $R^1$ of the form \eqref{JCF2}.  In this case, since $\lambda_1\ne0$, $R^1$ is nonsingular.  By \eqref{un2.7}, this implies that $R^1 = -L^1$ and $R^2 = -L^2$. So by \eqref{2.7}, we have $0=[R^1,R^2]$.  This implies that $R^2$ must be of the following form:
$$\begin{pmatrix}
	R_{11}	&	R_{12}	\\
	0		&	R_{11}
\end{pmatrix}.$$
By our requirement for nilindipendence, it must be true that $\lambda_1\ne0$ and $R_{11}\ne0$. However, if $\lambda_1\ne0$ and $R_{11}\ne0$, then $R^1$ and $R^2$ will not be nilindependent.  So, there are no Leibniz algebras $L(2,2)$ defined by $R^1$ of the form \eqref{JCF2}.

\subsection{$R^1$ diagonal}\label{case2}
We now consider $R^1$ of the form \eqref{JCF1}.  Note here that, since at least one of $\lambda_1$ or $\lambda_2$ must be nonzero, we will specify that $\lambda_1\ne0$.  Due to the freedom of $\lambda_2$, we 
further divide such algebras into three cases.  
\begin{enumerate}
	\item	$\lambda_2\ne0$ and $\lambda_2\ne\lambda_1$
    \item	$\lambda_2 = \lambda_1$
    \item	$\lambda_2 = 0$
\end{enumerate}


\subsubsection{Case 1:}\label{2.1.1}
Suppose that $\lambda_2\ne0$ and $\lambda_2\ne\lambda_1$.
Then, $R^1$ is invertible and we have that $R^1=-L^1$ and $R^2=-L^2$, by \eqref{un2.7}.  Then, 
by \eqref{2.7} and since $\lambda_1\ne\lambda_2$, 
$R^2$ is a diagonal matrix. That is, the products of $x_1$ and $x_2$ on $A(2)$ are defined, respectively, by the matrices
$$R^1 = \begin{pmatrix}\lambda_1 & 0 \\ 0 & \lambda_2\end{pmatrix} = -L^1 
\quad\text{and}\quad
R^2 = \begin{pmatrix}\mu_1 & 0 \\ 0 & \mu_2\end{pmatrix} = -L^2,$$
where we are denoting $R_{11}$ by $\mu_1$ and $R_{22}$ by $\mu_2$ for ease of notation.

Since $\lambda_1 \neq \lambda_2$ and $\{x_1, x_2\}$ are nilindependent, using \eqref{Gtransform} we can take a linear combination of $\{x_1, x_2\}$ so that
$$R^1 = \begin{pmatrix}1 & 0 \\ 0 & 0\end{pmatrix} = -L^1 
\quad\text{and}\quad
R^2 = \begin{pmatrix}0 & 0 \\ 0 & 1\end{pmatrix} = -L^2.$$

Applying Lemma \ref{nullspace}, we find that $\si^{\al\al} = \si^{\al\be}+\si^{\be\al} = 0$ for all $\al$, $\be$.
Therefore, the solvable Leibniz algebra $L(2,2)$ with products of $x_1$ defined by the above assumptions on $R^1$, is also isomorphic to a Lie algebra. 

\subsubsection{Case 2:}\label{2.1.2}
Suppose $\lambda_2=\lambda_1\ne 0$.  
Then $R^1$ is invertible, so by \eqref{un2.7} $R^1=-L^1$ and $R^2=-L^2$.
Since $R^1$ is a scalar multiple of the identity matrix, we can use \eqref{Stransform} to put $R^2$ in Jordan form and $R^1$ will be left invariant.  Then $R^2$ will be of form \eqref{JCF1} or \eqref{JCF2}, but if $(R^2)_{11} = (R^2)_{22}$, then $R^1$ and $R^2$ will not be nilindependent, a contradiction.  Hence $R^2$ is a diagonal matrix with $(R^2)_{11} \ne (R^2)_{22}$, so either $R^2$ or $R^2$ plus a multiple of $R^1$ will have $(R^2)_{11}$ and $(R^2)_{22} \neq 0$, so we obtain an algebra from Case 1 by interchanging $R^1$ and $R^2$. Thus we find no new Leibniz algebras in this case.

\subsubsection{Case 3:}\label{2.1.3}
Suppose $\lambda_2=0$.
Then \eqref{2.7twist} implies that $L^1$ and $L^2$ are diagonal.  Applying \eqref{un2.7} with $\alpha=2$, $\beta=1$ we find that $R^2$ is upper-triangular.  Since $R^1$ and $R^2$ are nilindependent this implies that $(R^2)_{22} \ne 0$.  Applying \eqref{un2.7} with other choices of $\alpha$ and $\beta$ we find that $R^1=-L^1$ and $R^2=-L^2$ (and $R^2$ is diagonal).  Again, either $R^2$ or $R^2$ plus a multiple of $R^1$ will have $(R^2)_{11}$ and $(R^2)_{22} \neq 0$, so we obtain an algebra from Case 1 and find no new Leibniz algebras in this case.

\subsection{General result for solvable Leibniz algebras $L(2,2)$}
By the allowable transformations \eqref{Stransform} and \eqref{Gtransform}, with the arguments given in \ref{case1} and \ref{case2}, we have a general result on all solvable Leibniz algebras of type $L(2,2)$.
\begin{thm}
All solvable Leibniz algebras $L(2,2)$, 2-dimensional extensions of a 2-dimensional abelian nilradical, are also Lie algebras.
\end{thm}


\noindent{\bf Acknowledgements.}  This work was completed as part of a undergraduate independent research project at Spring Hill College.  The authors would like to extend gratitutde to the support given by the Department of Mathematics at Spring Hill College.

\newpage

\begin{center}
\begin{small}
\begin{longtable}{lllcc}
No.	&	$R$	&	$L$		&	$\si$	&	restrictions	\endhead 
\hline\hline 
(1)
&	$
0$	&	$1$	&	
	$\si\in F$	&	\\
\hline
\hline
(2.1)	
&	$ 0 $

&	$\left(\begin{array}{cc}
	1&0\\
	0&0\\
	\end{array}\right)$
&	$\si_1=0,\si_2 \in F$	
&	\\
\hline
(2.2)	
&	$ 
0$	
&	$\left(\begin{array}{cc}
	1&0\\
	0&a\\
	\end{array}\right)$
&	$\si_1,\si_2=0$	
&	$a\ne0\in F$\\
\hline
(2.3)	
&	$ 
0$	
&	$\left(\begin{array}{cc}
	1&1\\
	0&1\\
	\end{array}\right)$
&	$\si_1,\si_2=0$	
&	\\
\hline
(2.4)	
&	$ 
\left(\begin{array}{cc}
	1&0\\
	0&0\\
\end{array}\right)$	
&	$\left(\begin{array}{cc}
	-1&0\\
	0&a\\
	\end{array}\right)$
&	$\begin{array}{c}
	\si_1\in F,	\\
    \si_2=0
    \end{array}$	
&	$a\ne0\in F$\\
\hline
(2.5)	
&	$ 
\left(\begin{array}{cc}
	1&0\\
	0&0\\
\end{array}\right)$	
&	$L=-R$
&	$\si_1, \si_2 \in F$, not both zero	
&		\\
\hline
\hline
(3.1) 
&	$0$	
&	$\left(\begin{array}{ccc}
	1&&\\
	&0&\\
	&&0\\
\end{array}\right)$
&	$\begin{array}{c}
	\si_{1}=0	\\
    \si_2,\si_3\in F
	\end{array}$
&		\\
\hline
(3.2) 
&	$0$	
&	$\left(\begin{array}{ccc}
	1&&\\
	&a&\\
	&&0\\
\end{array}\right)$
&	$\begin{array}{c}
	\si_{1},\si_2=0	\\
    \si_3\in F
	\end{array}$
&	$a\neq 0$	\\
\hline
(3.3) 
&	$0$	
&	$\left(\begin{array}{ccc}
	1&&\\
	&a&\\
	&&b\\
\end{array}\right)$
&	$\si_{1}, \si_2, \si_3=0$	
&	$a,b\neq 0 \in F$	\\
\hline
(3.4) 
&	$0$	
&	$\left(\begin{array}{ccc}
	1&1&\\
	&1&\\
	&&0\\
\end{array}\right)$
&	$\begin{array}{c}
	\si_{1}, \si_2=0	\\
    \si_3\in F
	\end{array}$
&		\\
\hline
(3.5) 
&	$0$	
&	$\left(\begin{array}{ccc}
	1&&\\
	&0&1\\
	&&0\\
\end{array}\right)$
&	$\begin{array}{c}
	\si_{1}, \si_3=0	\\
    \si_2 \in F
	\end{array}$
&	\\
\hline
(3.6) 
&	$0$	
&	$\left(\begin{array}{ccc}
	1&1&\\
	&1&\\
	&&a\\
\end{array}\right)$
&	$\si_{1}, \si_2, \si_3=0	$
&	$a\neq 0\in F$	\\
\hline
(3.7) 
&	$0$	
&	$\left(\begin{array}{ccc}
	1&&\\
	&a&1\\
	&&a\\
\end{array}\right)$
&	$\si_{1},\si_2, \si_3=0	$
&	$a \neq 0 $ or $1 \in F$	\\
\hline
(3.8) 
&	$0$	
&	$\left(\begin{array}{ccc}
	1&&\\
	&1&\\
	&&1\\
\end{array}\right)$
&	$\si_{1},\si_2,\si_3=0$
&		\\
\hline
(3.9)
& $\left(\begin{array}{ccc} 0&1&\\ &0&\\ &&0\\ \end{array}\right)$
&$\left(\begin{array}{ccc} 0&-1&a\\ 0&0&0\\ 0&0&1 \\\end{array}\right)$
& $\begin{array}{c}
	\si_3=0	\\
    \si_{1}, \si_2\in F
	\end{array}$
&	$a\in F$	\\
\hline
(3.10)
& $\left(\begin{array}{ccc} 0&1&\\ &0&\\ &&0\\ \end{array}\right)$
&$\left(\begin{array}{ccc} 0&a&b\\ 0&0&0\\ 0&0&1 \\\end{array}\right)$
& $\begin{array}{c}
	\si_2,\si_3=0	\\
    \si_{1} \in F
	\end{array}$
& 	$\begin{array}{c} a \neq -1 \in F\\ b\in F \end{array}$	\\
\hline
(3.11)
& $\left(\begin{array}{ccc} 0&1&\\ &0&\\ &&0\\ \end{array}\right)$
&$\left(\begin{array}{ccc} 0&a&b\\ 0&0&0\\ 0&1&1 \\\end{array}\right)$
& $\begin{array}{c}
	\si_2,\si_3=0	\\
    \si_{1}\in F
	\end{array}$
&	$a,b\in F$	\\
\hline
(3.12) 
&	$\left(\begin{array}{ccc}
	1&&\\
	&0&\\
	&&0\\
\end{array}\right)$	
&	$\left(\begin{array}{ccc}
	-1&&\\
	&a&\\
	&&b\\
\end{array}\right)$
&	$\begin{array}{c}
	\si_2,\si_3=0	\\
    \si_{1}\in F
	\end{array}$
&	$a,b\neq 0 \in F$	\\
\hline
(3.13) 
&	$\left(\begin{array}{ccc}
	1&&\\
	&0&\\
	&&0\\
\end{array}\right)$	
&	$\left(\begin{array}{ccc}
	-1&&\\
	&a&\\
	&&0\\
\end{array}\right)$
&	$\begin{array}{c}
	\si_2=0	\\
    \si_{1}, \si_3\in F
	\end{array}$
&	$a\neq 0 \in F$	\\
\hline
(3.14) 
&	$\left(\begin{array}{ccc}
	1&&\\
	&0&\\
	&&0\\
\end{array}\right)$	
&	$\left(\begin{array}{ccc}
	-1&&\\
	&a&1\\
	&&a\\
\end{array}\right)$
&	$\begin{array}{c}
	\si_2,\si_3=0	\\
    \si_{1}\in F
	\end{array}$
&	$a\neq 0 \in F$	\\
\hline
(3.15) 
&	$\left(\begin{array}{ccc}
	1&&\\
	&0&\\
	&&0\\
\end{array}\right)$	
&	$\left(\begin{array}{ccc}
	-1&&\\
	&0&1\\
	&&0\\
\end{array}\right)$
&	$\begin{array}{c}
	\si_3=0	\\
    \si_{1}, \si_2 \in F
	\end{array}$
&		\\
\hline
(3.16) 
&	$\left(\begin{array}{ccc}
	1&&\\
	&0&\\
	&&0\\
\end{array}\right)$	
&	$L=-R$
&	$\begin{array}{c}
	\si_{1}, \si_2, \si_3 \neq 0,	\\
	\si_2 \geq\si_3
	\end{array}$
&		\\
\hline
(3.17) 
&	$\left(\begin{array}{ccc}
	1&&\\
	&a&\\
	&&0\\
	\end{array}\right)$	
&	$\left(\begin{array}{ccc}
	-1&&\\
	&-a&\\
	&&b\\
	\end{array}\right)$
&	$\begin{array}{c} 
	\si_3=0\\
	\si_{1}, \si_2 \in F 
    \end{array}$
&	$\begin{array}{c}
	a \ne 0 \in F\\
	b \ne 0 \in F
	\end{array}$\\
\hline
(3.18) 
&	$\left(\begin{array}{ccc}
	1&&\\
	&a&\\
	&&0\\
	\end{array}\right)$	
&	$L=-R$
&	$\si_1,\si_2, \si_3\ne0\in F$
&	$a \ne 0 \in F$	\\
\hline
(3.19)
&	$\left(\begin{array}{ccc}
	1&1&\\
	&1&\\
	&&0\\
	\end{array}\right)$	
&	$\left(\begin{array}{ccc}
	-1&-1&\\
	&-1&\\
	&&a\\
	\end{array}\right)$
&	$\begin{array}{c}
 \si_3=0 \\
\si_{1},\si_2 \in F
\end{array}$
&	$a \ne 0 \in F$	\\
\hline
(3.20)
&	$\left(\begin{array}{ccc}
	1&1&\\
	&1&\\
	&&0\\
	\end{array}\right)$	
&	L = -R
&	$\si_1,\si_2, \si_3\ne0\in F	$
&		\\
\hline
(3.21)
&	$\left(\begin{array}{ccc}
	1&&\\
	&0&1\\
	&&0\\
	\end{array}\right)$	
&	$\left(\begin{array}{ccc}
	-1&&\\
	&0&a\\
	&&0\\
	\end{array}\right)$
&	$\begin{array}{c}
	\si_3=0	\\
    \si_{1},\si_2 \in F
	\end{array}$
&	$a\ne -1 \in F	$	\\
\hline
(3.22)
&	$\left(\begin{array}{ccc}
	1&&\\
	&0&1\\
	&&0\\
	\end{array}\right)$	
&	$L=-R$
&	$ \si_1, \si_2, \si_3 \neq 0 \in F $
&		\\
\hline\hline
\end{longtable}
\label{L31}\label{L11}\label{L21}
\end{small}
\end{center}

\end{document}